\documentclass[11pt,leqno]{article}  
\usepackage{amsmath,amstext, amsthm,amsfonts,amssymb,amsthm}  
 \numberwithin{equation}{section} 
\usepackage{epsfig}
\newtheorem{thm}{Theorem}[section]

\renewcommand{\O}{\mathcal{O}}
\newtheorem{cor}[thm]{Corollary}
\newcommand{\E}{{\cal E}_q}

\newcommand{\be}{\begin{equation}}
\newcommand{\ee}{\end{equation}}

\newcommand{\ba}{\begin{array}}
\newcommand{\ea}{\end{array}}

\newcommand{\bg}{\begin{gathered}}
\newcommand{\eg}{\end{gathered}}
\renewcommand{\r}{\rho}

\renewcommand{\b}{\beta}
\renewcommand{\t}{\theta}
\renewcommand{\O}{\mathcal{O}}

\newcommand{\f}{\phi}
\renewcommand{\r}{\rho}

\newcommand{\mc}{\mathcal}
\newcommand{\cqHp }{continuous $q$-Hermite polynomials }

\newcommand{\aw}{ Askey--Wilson  }

\newcommand{\bea}{\begin{eqnarray}}
\newcommand{\eea}{\end{eqnarray}}

\newcommand{\D}{\mathcal{D}_q}

\newcommand{\Sum}{\sum_{n=0}^\infty}
\renewcommand{\(}{\left( } \renewcommand{\)}{\right) }

\begin{document}

\title{$q$-Fractional Integral Operators With Two Parameters}

\author{ Mourad E. H.  Ismail 
 \\
 \and  Keru Zhou\thanks{Corresponding author}}

\maketitle

\begin{abstract}
We use the Poisson kernel of the continuous $q$-Hermite polynomials to introduces families of integral operators, which are 
semigroups of linear operators. We describe the eigenvalues and eigenfunctions of one family of operators. The action of the  semigroups of operators on the Askey--Wilson polynomials is shown to only change the parameters but preserves the degrees, hence we produce transmutation relation for the  Askey--Wilson polynomials. The transmutation relations are then used to derive bilinear generating functions involving the Askey--Wilson polynomials. 
\end{abstract}
  
  

Filename: Ism-ZhouPaperv3.tex

 \section{Introduction}
 The Askey-Wilson operator, $\D$, \cite{And:Ask:Roy} is a divided difference operator and is a $q$-analogue 
 of the differentiation operator $d/dx$. Several versions of a right 
 inverse of the Askey-Wilson operator were described in  
 \cite{Bro:Ism} \cite{Ism:Rah} \cite{Ism:Zha1}. The idea in these 
 works is to define $\D^{-1}$ on a basis for a weighted $L_2$ 
 space then extend it by linearity.  On the other hand, several 
 works defined $q$-fractional integral operators in different ways  
 see \cite{Ismbook}.  In  
  \cite{Ism:Zha:Zho},  the authors  introduced three one parameter 
  semigroups of  operators. 
 
 The purpose of this work is to generalize the  Ismail-Rahman 
 work  \cite{Ism:Rah2} by introducing a family $q$-fractional 
 operators with two parameters, which will be denoted by 
 $\{\mathcal{K}_{a,c}:a>0\}$. These operators have the composition 
 law $\mathcal{K}_{b,cq^{-a/2}}
 \mathcal{K}_{a,c}=\mathcal{K}_{a+b,c}$ which is similar 
 to the semigroup property.  These operators also have the 
 property $\D \mathcal{K}_{a,c}=\mathcal{K}_{a-1,c}$ if $a>1$. 
 In addition, we  identify an analogue of the infinitesimal generator 
 for  this family 
 of operators. We also introduce another more general family 
 and briefly discuss its properties. We refer the interested reader 
 to the work of Annaby and Mansour \cite{Ann:Man} which surveys  
 some classical $q$-fractional integrals.

In the sequel we shall follow the standard notations 
for $q$-series in \cite{And:Ask:Roy} \cite{Gas:Rah}, and   \cite{Ismbook}. In particular we use 
\begin{equation}
 h(\cos \t; a_1,a_2,\cdots,a_n):= 
  \prod_{j=1}^{n}(a_je^{i\t},a_je^{-i\t};q)_\infty.
\end{equation}
The  \cqHp  play an important role in the development of 
$q$-fractional integrals. 
They have the generating  function  
\begin{eqnarray}
\Sum \frac{H_n(\cos \t|q)}{(q;q)_n} t ^n = \frac{1}{(te^{i\t}, te^{-i\t};q)_\infty},
\label{eqGFHn}
\end{eqnarray}
and are generated by 
\begin{eqnarray}
\label{eqqHrr}
\bg
H_0(x|q) =1, \quad H_1(x|q) = 2x\\
2xH_n(x|q) = H_{n+1}(x|q) + (1-q^n) H_{n-1}(x|q). 
\eg
\end{eqnarray}
The  normalized weight function of the  \cqHp is 
\begin{eqnarray}
w_H(x|q) = \frac{(q, e^{2i\t}, e^{-2i\t};q)_\infty}{2\pi \sqrt{1-x^2}}, 
\end{eqnarray}
and the orthogonality  relation is 
\begin{eqnarray}
\label{eqorcqH}
\int_{-1}^1 H_m(x|q)H_n(x|q) w_H(x|q)  dx = (q;q)_n \delta_{m,n}.
\end{eqnarray}
The Poisson kernel for the \cqHp is \cite[Theorem 13.1.6]{Ismbook}
\begin{eqnarray}
\label{eqPoiKer}
\Sum H_n(\cos \t|q)  H_n(\cos \f|q) \frac{t^n}{(q;q)_n} = 
\frac{(t^2;q)_\infty}{(t e^{i(\t+\f)}, t e^{i(\t-\f)}, t e^{-i(\t+\f)}, 
	t e^{i(\f-\t)};q)_\infty}. 
\end{eqnarray}
See also  \cite{And:Ask:Roy}, \cite{Koe:Swa}.

To define the Askey-Wilson operator, we write the 
variable $x$ as 
$x = (z+1/z)/2$ and with any function $f(x)$, we let  $\breve{f}(z):= f(x)$.  The \aw operator is defined by 
\bea
\label{eqDefAWOP}
(\D f)(x) = \frac{\breve{f}(q^{1/2}z) - \breve{f}(q^{-1/2}z)}{(q^{1/2}-q^{-1/2})(z-1/z)/2}.
\eea

  Section 2 contains preliminary information.  In Section 3, 
we introduce the operators $\mathcal{K}_{a,c}$ and derive  their multiplication 
properties and identify and a one sided inverse operator. This is 
followed by Section 4, where we find an analogue of an  
infinitesimal generator for this family. The latter operator is defined  
on a basis in some specific $L_2$ space associated with \cqHp.  
We also establish  some general properties  for the generalized  
infinitesimal generator. In Section 5, we evaluate the action of the 
operator $\mathcal{K}_{a,c}$ on certain basis and $q$-exponential 
functions. In Section 6 we apply our results  to construct a bilinear 
symmetric kernel involving  special  Askey--Wilson polynomials. 

Finally Section 7 extends the results of \S 6 to a three parameter 
family of operators, and slightly more general Askey--Wilson 
polynomials. 

The Askey-Wilson integral, \cite{And:Ask:Roy}, \cite{Gas:Rah}. \cite{Ismbook}, is 
	\begin{eqnarray}
	\label{eqAWI}
	\int_0^\pi  
	\frac{w_H(\cos \f|q)}
	{h(\cos \f; a_1,a_2, a_3, a_4)} \sin \f d\f = 
	\frac{(a_1a_2a_3a_4;q)_\infty}{\prod_{1 \le j < k \le 4} 
		(a_ja_k;q)_\infty},
	\end{eqnarray}
and the Jacobi triple product identity \cite{Gas:Rah}, \cite{Ismbook}
\begin{eqnarray}\label{jtp}
(q,-z,-q/z;q)_\infty=\sum_{-\infty}^{\infty}q^{\binom{n}{2}}z^n.
\end{eqnarray}
 will be used throughout this work. 

\section{Preliminaries}
Ismail and Rahman \cite{Ism:Rah2} showed that a right inverse of Askey-Wilson operator $\D^{-1}$ may be defined by
\begin{eqnarray}\label{eqdefd}
\bg
(\D^{-1}  f)(\cos \t) = q^{-\frac{1}{2}} \left(\frac{1-q}{2c}\right) h(\cos \t;-cq^{1/2},-q^{1/2}/c)\\
\int_0^\pi \frac{w_H(\cos \f |q)(q;q)_\infty f(\cos\f)sin\f d\f}{h(\cos\f;q^{1/2}e^{i\t},q^{1/2}e^{-i\t})h(\cos\f;-1/c,-cq)}.
\eg
\end{eqnarray}
The original definition in \cite{Ism:Rah2} had an additional constant term but following \cite{Ism:Zha:Zho} we shall delete it. We emphasis an important fact that the strong operator limit $\lim_{p\rightarrow q^-}\mathcal{D}_p\mathcal{D}_q^{-1}$ is still the identity operator.

It is known that Korovin's theorem,  \cite{Kor}, can be used to prove the following result.
\begin{thm}\label{Lem1}Let $p_n$ be orthonormal with respect to a 
	probability measure $\mu$ on a compact  interval $[a,b]$. If the 
	Poisson kernel  $P_r(x,y):= \Sum p_n(x)p_n(y) r^n$ is $\ge 0$ for all $x, y \in$ supp$(\mu) \subset  [a,b]$,  then 
	\begin{eqnarray}
	\lim_{r \to 1^-} \int_a^b  P_r(x,y) f(y) d\mu(y) = f(x)
	\notag
	\end{eqnarray}
	is uniform  for all continuous functions on $[a,b]$.
\end{thm}

\section{$q$-Fractional Integral  Operators }
Motivated by  \eqref{eqdefd}, we define the following integral operator
\begin{eqnarray}\label{eqdefTa}
\bg
(\mathcal{K}_{a,c}  f)(\cos \t) = q^{\frac{a(a-3)}{4}} \left(\frac{1-q}{2c}\right)^a h(\cos \t;-cq^{1-a/2},-q^{a/2}/c)\\
\int_0^\pi \frac{w_H(\cos \f |q)(q^a;q)_\infty f(\cos\f)sin\f d\f}{h(\cos\f;q^{a/2}e^{i\t},q^{a/2}e^{-i\t})h(\cos\f;-1/c,-cq)}.
\eg
\end{eqnarray}

\begin{thm}\label{thm1}
	The  operators  $\{\mathcal{K}_{a,c}: a >0\}$ have the following properties. 
	
	\noindent  $\textup{(a)}$   $\mathcal{K}_{b,cq^{-a/2}} \mathcal{K}_{a,c} = \mathcal{K}_{a+b,c}$.
	
	\noindent   $\textup{(b)}$ $\mathcal{K}_{a,c}$ tends to the identity operator 
	as $a \to 0^+$. 
	
	\noindent $\textup{(c)}$  The action of $\D$ on 
	$\mathcal{K}_{a,c}$ is given by  $\D \mathcal{K}_{a,c} = \mathcal{K}_{a-1,c}$. 
\end{thm}
\begin{proof}[Proof of $\textup{(a)}$]
	It is easy to see that  
	\begin{eqnarray}
	\notag
	\bg
	q^{(a+b)(3-a-b)/4}  \left(\frac{2c}{1-q}\right)^{a+b}	\frac{(\mathcal{K}_{b,cq^{-a/2}} \mathcal{K}_{a,c}f)(\cos\t)}{h(\cos \t;-cq^{1-(a+b)/2},-q^{(a+b)/2}/c)}\\
	=  (q^a, q^b; q)_\infty \int_{[0,\pi]^2} 
	\frac{ w_H(\cos \f|q)}
	{h(\cos \f; q^{a/2} e^{i\t}, q^{a/2} e^{-i\t})}      \\
	\times  \frac{ w_H(\cos \psi|q) 
		f(\cos \psi) \sin \f \sin \psi } {
		h(\cos \psi; q^{b/2}e^{i\f}, q^{b/2} e^{-i\f})h(\cos\psi;-1/c,-cq)} \;d\f d\psi.
	\eg
	\end{eqnarray}
	The  completeness of the  \cqHp and the Poisson kernel \eqref{eqPoiKer} imply
	\begin{eqnarray}
	\notag
	\bg
	\int_0^\pi \frac{(q^a, q^b; q)_\infty w_H(\cos \f|q)\sin \f \; d\f}
	{h(\cos \psi; q^{b/2}e^{i\f}, q^{b/2} e^{-i\f})
		h(\cos \f; q^{a/2} e^{i\t}, q^{a/2} e^{-i\t})} \\
	=  \frac{(q^{a+b};q)_\infty} {h(\cos \psi; q^{(a+b)/2}e^{i\t}, 
		q^{(a+b)/2} e^{-i\t})}
	\eg
	\end{eqnarray}
	and (a) holds.  
\end{proof}
\begin{proof}[Proof of $\textup{(b)}$]
We know that 	$f(\cos\t)/h(\cos\t;-1/c,-cq)$ is continuous if $f(\cos\t)$ is continuous.   Now,  Theorem \ref{Lem1} implies part (b).
\end{proof}
\begin{proof}[Proof of $\textup{(c)}$]
 Use direct computation. 
\end{proof}

Next, we define  a left-inverse for $\mathcal{K}_{a,c}$. Recall the 
notation $\lfloor{a}\rfloor$ and $\{a\}$ mean the integer part (floor) 
and the fractional part of $a$. 

\begin{thm} The following identity holds 

\bea
\notag
\D^{\lfloor{a}\rfloor+1} K_{1-\{a\},cq^{-a}}K_{a,c}=I.
\eea
	where $I$ is identity map.
\end{thm}
\begin{proof}
	It is obvious that if $u(x)=\mathcal{K}_{a,c}f(x)$ then
	\begin{eqnarray}
	\bg
	\notag
	\D^{\lfloor{a}\rfloor+1} K_{1-\{a\},cq^{-a}}u(x)= \D^{\lfloor{a}\rfloor+1} K_{1-\{a\},cq^{-a}}K_{a,c}f(x)=\D^{\lfloor{a}\rfloor+1}K_{[a]+1,c}f(x).
	\eg
	\end{eqnarray}
	In fact, $\mathcal{K}_{n,c}$ is a right inverse for $\D^{n}$ if $n$ is positive integer, so $\D^{[a]+1}K_{\lfloor{a}\rfloor+1,c}f(x)=f(x)$ and the proof is complete.
\end{proof}

\section{An Infinitesimal Generator}
In analogy with the theory of semigroups,  we define 

\bea
J_t(a,c)=\lim_{\triangle a\rightarrow 0^+}\frac{K_{a+\triangle a,c}-K_{a,c}}{\triangle a}.  
\label{eqdefJ}
\eea
This is an analogue of the infinitesimal generator for a semigroup. 
An important property  of the infinitesimal generator is that its action 
on an element of the semigroup $\{S(t)\}$ is at time $t$ is $S(t)$ times 
its action at $t=0$. Our generalization has an analogous property, 
see Theorem 4.4. 

\begin{thm}\label{qhermite}
The action of	$\mathcal{K}_{a,c}$  on $h(\cos\f;-1/c,-cq) H_n(\cos \f;q)$ is given by 
	\begin{eqnarray}
	\label{eqKhH}
	\bg		
	(\mathcal{K}_{a,c}h(\cos\f;-1/c,-cq) H_n(\cos \f;q))(cos\t)\\
	=q^{a(a-3)/4+na/2}\; \left(\frac{1-q}{2c}\right)^a h(\cos \t;-cq^{1-a/2},-q^{a/2}/c)H_n(\cos \t|q)
	\eg
	\end{eqnarray} 
	where $\{H_n(\cos \t|q)\}$ are \cqHp .
\end{thm}
\begin{proof}
It is clear that 
	\begin{eqnarray}
	\notag
	\bg
	\int_0^\pi  
	\frac{(q^a;q)_\infty h(\cos \t;-cq^{1-a/2},-q^{a/2}/c) H_m(\cos \f;q) w_H(\cos \f|q)} 
	{h(\cos \f; q^{a/2}e^{i\t}, q^{a/2} e^{-i\t})} \sin \f d\f \\
	=h(\cos \t;-cq^{1-a/2},-q^{a/2}/c)   \int_{0}^\pi H_m(\cos \f|q) 
	w_H(\cos \f|q)
	\Sum H_n(\cos \t|q)H_n(\cos \f|q) \frac{q^{qn/2}}{(q;q)_n}\sin\f \;d\f\\
	= h(\cos \t;-cq^{1-a/2},-q^{a/2}/c)H_m(\cos \t|q)q^{am/2}.
	\eg
	\end{eqnarray}
\end{proof}

\begin{thm}
	The operator $J$ of \eqref{eqdefJ} satisfies  
	\begin{eqnarray}
	\bg
	\notag
	(J_t(a,c)h(\cos\f;-1/c,-cq)H_n(\cos \f|q))(\cos \t)\\
	=q^{a(a-3)/4+na/2} \left(\frac{1-q}{2c}\right)^a\Big[ h(\cos \t;-cq^{1-a/2},-q^{a/2}/c)\log \Big(\frac{(1-q)q^{(n+1)a/2-3/4}}{2c}\Big)\\ 
	+\left(\sum_{-\infty}^\infty \frac{q^{\binom{k}{2}}(zq^{a/2}/c)^k}{(q;q)_\infty}\frac{k\log q}{2}\right)(-q^{a/2}/cz,-czq^{1-a/2};q)_\infty\\+\left(\sum_{-\infty}^\infty\frac{q^{\binom{k}{2}}(q^{a/2}/zc)^k}{(q;q)_\infty}\frac{k\log q}{2}\right)(-q^{a/2}z/c,-cq^{1-a/2}/z;q)_\infty \Big] H_n(\cos \t|q)
	\eg
	\end{eqnarray}
	where $z=e^{i\t}$.
\end{thm}
\begin{proof}
	From Theorem \ref{qhermite} and \eqref{jtp}, it is clear that 
	\begin{eqnarray}
\bg
\notag
\mathcal{K}_{a+\triangle a,c}h(\cos\f;-1/c,-cq)H_n(\cos \f|q)\\
=q^{(a+\triangle a)(a+\triangle a-3)/4+n(a+\triangle a)/2} \left(\frac{1-q}{2c}\right)^{a+\triangle a}h(\cos \t;-cq^{1-(a+\triangle a)/2},-q^{(a+\triangle a)/2}/c)H_n(\cos \t|q)\\
=q^{a(a-3)/4+na/2}  \left(\frac{1-q}{2c}\right)^a\left[1+ \triangle a \log \left(\frac{(1-q)q^{(2a-3)/4+na/2}}{2c}\right)+\O(\triangle a)^2\right]\\
\times \left[\sum_{-\infty}^\infty \frac{q^{\binom{k}{2}}(zq^{a/2}/c)^k}{(q;q)_\infty}+\triangle a\frac{q^{\binom{k}{2}}(zq^{a/2}/c)^k}{(q;q)_\infty}\frac{k\log q}{2}+\O(\triangle a)^2\right]\\
\times \left[\sum_{-\infty}^\infty \frac{q^{\binom{k}{2}}(q^{a/2}/zc)^k}{(q;q)_\infty}+\triangle a\frac{q^{\binom{k}{2}}(q^{a/2}/cz)^k}{(q;q)_\infty}\frac{k\log q}{2}+\O(\triangle a)^2\right].\\ 
\eg
\end{eqnarray}
	In fact, we can rewrite the above  as	
	\begin{eqnarray}
	\bg
	\notag
	q^{a(a-3)/4+na/2} \left(\frac{1-q}{2c}\right)^a h(\cos \t;-cq^{1-a/2},-q^{a/2}/c)H_n(\cos \t|q)
	+ \triangle aq^{a(a-3)/4+na/2}\\\times\left(\frac{1-q}{2c}\right)^a\Big[ h(\cos \t;-cq^{1-a/2},-q^{a/2}/c)\log \Big(\frac{(1-q)q^{(n+1)a/2-3/4}}{2c}\Big)\\ 
	+\left(\sum_{-\infty}^\infty\frac{q^{\binom{k}{2}}(zq^{a/2}/c)^k}{(q;q)_\infty}\frac{k\log q}{2}\right)(-q^{a/2}/cz,-czq^{1-a/2};q)_\infty\\+\left(\sum_{-\infty}^\infty\frac{q^{\binom{k}{2}}(q^{a/2}/zc)^k}{(q;q)_\infty}\frac{k\log q}{2}\right)(-q^{a/2}z/c,-cq^{1-a/2}/z;q)_\infty \Big] H_n(\cos \t|q)+\O(\triangle a)^2.
	\eg
	\end{eqnarray}
\end{proof}
\begin{cor}
	The case $a=0$ of infinitesimal generator is
	\begin{eqnarray}
	\bg
	\notag
	(J_t(0,c)h(\cos\f;-1/c,-cq)H_n(\cos \f|q))(\cos \t)\\
	=h(\cos \t;-cq,-1/c)H_n(\cos \t|q)\log \left(\frac{(1-q)q^{-3/4}}{2c}\right)\\
	+\left[\left(\sum_{-\infty}^\infty\frac{q^{\binom{k}{2}}(z/c)^k}{(q;q)_\infty}\frac{k\log q}{2}\right)(-1/cz,-czq;q)_\infty+\left(\sum_{-\infty}^\infty\frac{q^{\binom{k}{2}}(1/zc)^k}{(q;q)_\infty}\frac{k\log q}{2}\right)(-z/c,-cq/z;q)_\infty \right]H_n(\cos\t|q)		
	\eg
	\end{eqnarray}
	where $z=e^{i\t}$.	
\end{cor}
\begin{thm}
	We have the following relation for infinitesimal generator
	$$J_t(a,c)=J_t(0,cq^{-a})\mathcal{K}_{a,c}$$
\end{thm}
\begin{proof}
	Note the define of $J_t(a,c)$ and Theorem \ref{thm1} \textup{(a)}
	\begin{eqnarray}
	\bg
	\notag
	J_t(a,c)=\lim_{\triangle a \rightarrow 0^+} \frac{\mathcal{K}_{a+\triangle a,c}-\mathcal{K}_{a,c}}{\triangle a}=\lim_{\triangle a \rightarrow  0^+} \frac{\mathcal{K}_{\triangle a,cq^{-a}}-I}{\triangle a}\mathcal{K}_{a,c}=J_t(0,cq^{-a})\mathcal{K}_{a,c}. 
	\eg
	\end{eqnarray}
\end{proof}

\section{Analogues of Monomials and Exponentials}
There are three natural $q$-analogues of the monomials, namely
\begin{eqnarray} \label{eqDeffn&rn}
\bg
\phi_n(x)  =  (q^{1/4}e^{i\t}, q^{1/4} e^{-i\t};q^{1/2})_n =
\prod_{k=0}^{n-1} [1- 2xq^{1/4+k/2} + q^{1/2+ k}] \label{eqdefphin1/4}\\
\rho_n(x) = (1+e^{2i\t}) e^{-in \t}(-q^{2-n} e^{2i\t};q^2)_{n-1}, n>0, \quad \r_0(x):=1\\
\phi_n(x;a)=(ae^{i\t},ae^{-i\t};q)_n=\prod_{k=0}^{n-1} [1- 2axq^{k} + q^{2k}]
\eg
\end{eqnarray}
Ismail and Stanton  introduced the bases $\{\phi_n(x)\}$ and 
$\{\rho_n(x)\}$ in \cite{Ism:Sta1} and \cite{Ism:Sta2}, but the polynomials $\{\phi_n(x;a)\}$ go back to the original work of Askey and Wilson \cite{Ask:Wil}.  Ismail and Stanton 
and established the theory of $q$-Taylor series expansions on entire functions in these bases. Following the convention 
\bea
(A;q)_\b := \frac{(A;q)_\infty}{(Aq^\b;q)_\infty}
\eea
we extend the definitions of the above functions to the case when $n$ is not necessarily a positive integer. 
\begin{thm} We have 
	\begin{eqnarray}
	\notag
	\mathcal{K}_{a,c}\phi_\beta(x;-1/c)=q^{a(a-3)/4}(\frac{1-q}{2c})^a \frac{(q^{a+\b+1};q)_\infty}{(q^{\b+1};q)_\infty}\phi_{\b}(x;-q^{a/2}/c) \phi_{a}(x;-cq^{1-a/2}).
	\end{eqnarray}
\end{thm}
\begin{proof}
	It is clear that the action of $\mathcal{K}_{a,c}$   on $\phi_\beta(x;-1/c)$ is given by 
	\begin{eqnarray}
	\notag
	\bg
	q^{a(a-3)/4}(\frac{1-q}{2c})^a\; (q^a;q)_\infty \\
	\times  \int_0^\pi  
	\frac{h(\cos \t; -cq^{1-a/2},-q^{a/2}/c) w_H(\cos \f|q)\; \sin \f \;d\f}
	{(-q^{\b} e^{i\f}/c, -q^{\b} e^{-i\f}/c,-cqe^{i\f},-cqe^{-i\f};q)_\infty  h(\cos \f; q^{a/2}e^{i\t}, q^{a/2} e^{-i\t})} \\
	= q^{a(a-3)/4}(\frac{1-q}{2c})^a\; (q^a;q)_\infty \\
	\times  \int_0^\pi  
	\frac{h(\cos \t; -cq^{1-a/2},-q^{a/2}/c) w_H(\cos \f|q)\;  \sin \f \;d\f}
	{h(\cos \f; -cq, -q^{\b}/c, q^{a/2}e^{i\t}, q^{a/2} e^{-i\t})}.
	\eg 
	\end{eqnarray}
	The result then  follows from the evaluation of the Askey--Wilson integral, \eqref{eqAWI}.
	\end{proof}
In fact, using the same method, we  prove that  
\begin{eqnarray}
\bg
\mathcal{K}_{a,c}\phi_{\b}(x;-cq)=q^{a(a-3)/4}(\frac{1-q}{2c})^a\; \frac{(q^{a+\b+1};q)_\infty}{(q^{\b+1};q)_\infty}\phi_{a+\beta}(x;-cq^{1-a/2}).
\eg
\end{eqnarray}

The 
$q$-exponential function $\E$ from \cite{Ism:Zha1}, is defined via 
\begin{eqnarray}
\label{eqqExp}
\bg
\quad  \mc{E}_q(\cos\theta;\alpha) =\frac{\(\alpha^2;q^2\)_\infty}
{\(q\alpha^2;q^2\)_\infty} 
\Sum\(-ie^{i\theta}q^{(1-n)/2},-ie^{-i\theta}q^{(1-n)/2};q\)_n 
\frac{(-i\alpha)^n}{(q;q)_n}\,q^{n^2/4}. 
\eg
\end{eqnarray}
Its expansion in   \cqHp  is given by   
	\begin{eqnarray}
	\label{eqqExinqH}
	(qt^2;q^2)_\infty\E(x;t)= \Sum\frac{q^{n^2/4}t^n}{(q;q)_n}H_n(x\,|\, q), 
	\end{eqnarray} 
\cite{Ism:Zha1}
\begin{thm}
	The operators $\mathcal{K}_{a,c}$ have the property 
	\begin{eqnarray}
	\bg
	\notag
	(qt^2;q)_\infty\left(\mathcal{K}_{a,c}h(\cos \t;-1/c,-cq)\E(\cos \t;t)\right) \\
	=q^{a(a-3)/4}(\frac{1-q}{2c})^a(q^{a+1}t^2;q)_\infty h(\cos \t;-q^{a/2}/c,-cq^{1-a/2})\E(\cos \t;tq^{a/2}).
	\eg
	\end{eqnarray}
\end{thm}
\begin{proof}
	
We apply Theorem \ref{qhermite} and \eqref{eqqExinqH}, then
	\begin{eqnarray}
	\bg
	\notag
	\mathcal{K}_{a,c}h(\cos \t;-1/c,-cq)\E(\cos \t;t)\\=q^{a(a-3)/4}(\frac{1-q}{2c})^ah(\cos \t;-cq^{1-a/2},-q^{a/2}/c)\;\Sum\frac{q^{n^2/4}t^n}{(q;q)_n}H_n(x\,|\, q)q^{na/2},
	\eg
	\end{eqnarray}
	and  apply \eqref{eqqExinqH} again.
\end{proof}
We  can also prove the  important property 
\begin{eqnarray}\label{Taex}
\bg
\int_{-1}^1 \frac{\E(x;t)(\mathcal{K}_{a,c}f)(x) w_H(x|q)dx}{h(x;-cq^{1-a/2},-q^{a/2}/c)}\\=q^{a(a-3)/4}\left(\frac{1-q}{2c}\right)^a
\frac{(q^{a+1}t^2;q)_\infty}{(qt^2;q)_\infty}\int_{-1}^{1}\frac{\E(x;tq^{a/2})f(x) w_H(x|q)dx}{h(x;-cq,-1/c)}.	
\eg
\end{eqnarray}
\begin{proof}
	The result of multiplying  the  left-side of \eqref{Taex} by $(qt^2;q)_\infty$  is 
	\begin{eqnarray}
	\notag
	\bg
	q^{a(a-3)/4}\left(\frac{1-q}{2c}\right)^a \int_{-1}^{1}\int_{-1}^{1}
	\frac{\E(x;t)w_H(x|q)w_H(y|q)f(y)}{h(cos\f;-cq,-1/c)h(\cos \t;q^{a/2}e^{i\f},q^{a/2}e^{-i\f})}dy dx. 	
	\eg
	\end{eqnarray}
Interchange the order of integral, and apply  \eqref{eqqExinqH} and the Poisson Kernel of  \cqHp to see that the above quantity is
	\begin{eqnarray}
	\bg
	\notag
	q^{a(a-3)/4}(\frac{1-q}{2c})^a(qt^2;q)_\infty\int_{-1}^{1}\int_{-1}^{1}
	\(\Sum\frac{q^{n^2/4}t^n}{(q;q)_n}H_n(x\,|\, q)\)\(\Sum \frac{H_n(x|q)H_n(y|q)q^{an/2}}{(q;q)_n}\)\\ 
	\times \frac{w_H(x|q)w_H(y|q)f(y)}{h(cos\f;-cq,-1/c)} dxdy      \\
	=q^{a(a-3)/4}(\frac{1-q}{2c})^a\int_{-1}^{1} \Sum \frac{q^{n^2/4+an/2}t^n}{(q;q)_n} H_n(y|q) \frac{w_H(y|q)f(y)}{h(cos\f;-cq,-1/c)}dy\\
	=q^{a(a-3)/4}\left(\frac{1-q}{2c}\right)^a(q^{a+1}t^2;q)_\infty\int_{-1}^{1} \frac{\E(y;tq^{a/2})w_H(y|q)f(y)}{h(cos\f;-cq,-1/c)}dy.
	\eg
	\end{eqnarray}
	This  completes the proof.
\end{proof}

\section{Symmetric Bilinear  Kernels }
We recall the definition of the Askey--Wilson polynomials
\cite{Ismbook},  
\begin{eqnarray}
\bg
p_n(\cos \t; a, b ,c ,d) = (ab, ac, ad; q)_na^{-n} \qquad \\
\qquad \times 
{}_4\f_3\left(\left. \ba{c} 
q^{-n}, q^{n-1} abcd,  ae^{i\t}, ae^{-i\t}
\\
ab, \quad ac, \quad ad
\ea \right|q, q \right).
\eg
\end{eqnarray}
\begin{thm}
	The operator $\mathcal{K}_{a,c}$ maps an\aw polynomial 
	to a ${}_5\f_4$ function as follows 
	\begin{eqnarray}
	\label{eq5phi4}
	\bg
	\mathcal{K}_{a,c} p_n(.; -1/c, a_2, a_3, a_4) = (-1)^nq^{a(a-3)/4}c^{n}(\frac{1-q}{2c})^a 
	(-a_2/c, - a_3/c,-a_4/c;q)_n
	\\
	\times   
	\frac{( q^{a+1};q)_\infty}
	{(q;q)_\infty}\;  \frac{h(\cos \t; -q^{1-a/2}c)}{h(\cos \t; -cq^{1+a/2})}  \\
	\qquad \times 
	{}_5\f_4\left(\left. \ba{c} 
	q^{-n}, -q^{n-1} a_2a_3a_4/c, q,  -q^{a/2} e^{i\t}/c, -q^{a/2}e^{-i\t}/c
	\\
	-a_2/c, \quad - a_3/c, \quad -a_4/c,  \quad q^{a+1}
	\ea \right|q, q \right).
	\eg
	\end{eqnarray}
\end{thm}
\begin{proof}
	It is obvious that
	\begin{eqnarray}
	\bg
	\notag
	\frac{(-1)^nq^{-a(a-3)/4}c^{-n}(\frac{2c}{1-q})^a}{(-a_2/c, - a_3/c,-a_4/c;q)_n(q^a;q)_\infty h(\cos\t; -cq^{1-a/2},-q^{a/2}/c)}	T_a p_n(.; -1/c, a_2, a_3, a_4)\\=	\Sum \frac{(q^{-n},-q^{n-1}a_2a_3a_4/c;q)_k}{(q,-a_2/c,-a_3/c,-a_4/c;q)_k}  \int_0^\pi \frac{w_H(\cos \f |q) sin\f d\f}{h(\cos\f;q^{a/2}e^{i\t},q^{a/2}e^{-i\t})h(\cos\f;-q^k/c,-cq)} 
	\eg
	\end{eqnarray}
	The above integral is an Askey-Wilson integral, so we use the evaluation \eqref{eqAWI} and establish \eqref{eq5phi4} after some simplifications.  
\end{proof}

The same method of proof establishes the following identity  \begin{eqnarray}
\bg
(\mathcal{K}_{a,c} p_n(.; -cq, a_2, a_3, a_4))(\cos \t) = (-1)^nq^{a(a-3)/4}(cq)^{-n}(\frac{1-q}{2c})^a
\\
\times   
(-a_2cq, - a_3cq,-a_4cq;q)_n \; \frac{( q^{a+1};q)_\infty}
{(q;q)_\infty}\;  \frac{h(\cos \t; -q^{1-a/2}c)}{h(\cos \t; -cq^{1+a/2})}  \\
\qquad \times 
{}_5\f_4\left(\left. \ba{c} 
q^{-n}, -q^{n} ca_2a_3a_4, q,  -cq^{a/2+1} e^{i\t}, -cq^{a/2+1}e^{-i\t}
\\
-a_2cq, \quad - a_3cq, \quad -a_4cq,  \quad q^{a+1}
\ea \right|q, q \right),
\eg
\end{eqnarray}
whose proof will be omitted.

The special case $a_2=-cq$ is of interest  because the ${}_5\f_4$ 
becomes a ${}_4\f_3$, namely
\begin{eqnarray}
\label{eq4phi3}
\bg
\mathcal{K}_{a,c} p_n(.; -1/c, -cq, a_3, a_4) = (-1)^nq^{a(a-3)/4}c^{n}(\frac{1-q}{2c})^a 
\\
\times  (q, - a_3/c,-a_4/c;q)_n \;  
\frac{( q^{a+1};q)_\infty}
{(q;q)_\infty}\;  \frac{h(\cos \t; -q^{1-a/2}c)}{h(\cos \t; -cq^{1+a/2})}  \\
\qquad \times 
{}_4\f_3\left(\left. \ba{c} 
q^{-n}, q^{n} a_3a_4,   -q^{a/2} e^{i\t}/c, -q^{a/2}e^{-i\t}/c
\\
- a_3/c, \quad -a_4/c,  \quad q^{a+1}
\ea \right|q, q \right).
\eg
\end{eqnarray}
The ${}_4\f_3$ is indeed an Askey-Wilson polynomial and we have proved the transmutation relation 
\begin{eqnarray}
\label{taasp}
\bg
(\mathcal{K}_{a,c} p_n(.; -1/c, -cq, a_3, a_4))(\cos \t)  =q^{a(a-3)/4+na/2}(\frac{1-q}{2c})^a 
\frac{(q;q)_n}{(q^{a+1};q)_n}
\\
\times   
\frac{( q^{a+1};q)_\infty}
{(q;q)_\infty}\;  \frac{h(\cos \t; -q^{1-a/2}c)}{h(\cos \t; -cq^{1+a/2})} p_n(x; -q^{a/2}/c, -cq^{1+a/2}, q^{-a/2}a_3,  q^{-a/2}a_4).  
\eg
\end{eqnarray}

Ismail has shown in \cite{Ism} that transmutation relations lead to  bilinear formulas  through the  Hilbert-Schmidt theory of integral equations.  
The orthogonality relation of the  Askey-Wilson polynomial is,  \cite{And:Ask:Roy}, \cite{Gas:Rah}, \cite{Ask:Wil}    
\begin{eqnarray}\label{eqawpo}
\bg
\int_{0}^{\pi}p_m(\cos\t;\textbf{t}|q)p_n(\cos\t;\textbf{t}|q)w(\cos\t;\textbf{t})d\t\\
=\frac{2\pi(t_1t_2t_3t_4q^{2n};q)_{\infty}(t_1t_2t_3t_4q^{n-1};q)_{n}}{(q^{n+1};q)_{\infty}\prod_{1\le j<k\le 4}(t_jt_kq^{n};q)_{\infty}}\delta_{m,n}
\eg
\end{eqnarray}
where ${\bf t}= (t_1, t_2, t_3, t_4)$, and the weight function is 
\begin{eqnarray}
w(\cos\t;t_1,t_2,t_3,t_4): =\frac{(e^{2i\t},e^{-2i\t};q)_{\infty}}{\prod_{j=1}^4(t_je^{i\t},t_je^{-i\t};q)_{\infty}}.
\end{eqnarray}
For convenience  we set  
\begin{eqnarray}
\notag
\bg
M_n(t_1,t_2,t_3,t_4)  =\frac{2\pi(t_1t_2t_3t_4q^{2n};q)_{\infty}(t_1t_2t_3t_4q^{n-1};q)_{n}}{(q^{n+1};q)_{\infty}\prod_{1\le j<k\le 4}(t_jt_kq^{n};q)_{\infty}}\\
A_n=M_n(-q^{a/2}/c,-q^{1+a/2}c,q^{-a/2}a_3,q^{-a/2}a_4)\\
B_n=M_n(-1/c,-cq,a_3,a_4),\qquad 
C_n=\frac{(q^{a+n+1};q)_\infty}{(q^{n+1};q)_\infty}q^{an/2}.
\eg
\end{eqnarray}
Using \eqref{taasp}  and the orthogonality relation \eqref{eqawpo} 
 we find that  
\begin{eqnarray}\label{int1}
\bg
\int_{0}^{\pi}w_0(\cos\t|q )\int_0^\pi  
\frac{ w_H(\cos \phi_1|q) p_n(\cos\f_1|-1/c,-cq,a_3,a_4)}
{h(\cos \f_1 ;-1/c,-cq)  h(\cos \f_1; q^{a/2}e^{i\t}, q^{a/2} e^{-i\t})} \sin \f_1 \\
\int_0^\pi  
\frac{ w_H(\cos \f_2|q) p_m(\cos\f_2|-1/c,-cq,a_3,a_4)}
{h(\cos \f_2 ;-1/c,-cq))_\infty  h(\cos \f_2; q^{a/2}e^{i\t}, q^{a/2} e^{-i\t})} \sin \f_2 d\f_2 d\f_1 d\t\\
=A_nC_n^2 \delta_{m,n},
\eg
\end{eqnarray}
where 
\begin{eqnarray}
\notag
w_0(\cos\t|q)=w(\cos\t;-q^{a/2}/c,-q^{1+a/2}c,q^{-a/2}a_3,q^{-a/2}a_4)h^2(\cos \t; -cq^{1+a/2},-q^{a/2}/c)\;.
\end{eqnarray}
In other words we have 
\begin{eqnarray}\label{int2}
\bg
\int_0^\pi \frac{w_H(\cos \f_1|q)p_n(\cos\f_1|-1/c,-cq,a_3,a_4)}{h(\cos \f_1 ;-1/c,-cq)}
\sin \f_1 \\
\int_0^\pi  
\frac{ w_H(\cos \f_2|q) p_m(\cos\f_2|-1/c,-cq,a_3,a_4)}
{h(\cos \f_2 ;-1/c,-cq) } \sin \f_2 \\
\int_0^\pi  \frac{w_0(\cos\t|q )}{h(\cos \f_1; q^{a/2}e^{i\t}, q^{a/2} e^{-i\t})h(\cos \f_2; q^{a/2}e^{i\t}, q^{a/2} e^{-i\t})} d\t d\f_1 d\f_2\\
=A_nC_n^2 \delta_{m,n}.
\eg
\end{eqnarray}
The completeness of the Askey-Wilson polynomials in the corresponding weighted $L_2$ spaces lead to the transmutation relation    
\begin{eqnarray} \label{kernel}
\bg
\frac{A_nC_n^2}{B_n} p_n(\cos\f_2|-1/c,-cq,a_3,a_4)\\
=\int_{0}^{\pi}K(\cos\f_1,\cos\f_2)
p_n(\cos\f_1|-1/c,-cq,a_3,a_4)d\f_1,
\eg
\end{eqnarray}
where the kernel  $K(\cos \f_1,\cos \f_2)$ is defined by 
\begin{eqnarray}
\bg
K(\cos \f_1,\cos \f_2)=\frac{W_H(\cos \f_1|q)}{h(\cos \f_1 ;-1/c,-cq)} \frac{W_H(\cos \f_2|q)}{h(\cos \f_2 ;-1/c,-cq)}\\
\times \frac{\sin \f_1 \sin \f_2}{w(\cos \f_2;-1/c,-cq,a_3,a_4)}\\
\int_0^\pi  \frac{w_0(\cos\t|q )}{h(\cos \f_1; q^{a/2}e^{i\t}, q^{a/2} e^{-i\t})h(\cos \f_2; q^{a/2}e^{i\t}, q^{a/2} e^{-i\t})} d\t.
\eg	
\end{eqnarray}
The connection  relation \eqref{kernel} establishes  the  bilinear formula
\begin{eqnarray}\label{kernel2}
\bg
\frac{K(\cos\f_1,\cos\f_2)}{w(\cos\f_1|-1/c,-cq,a_3,a_4)} \\
= \sum_{n=0}^{\infty} A_n\left(\frac{C_n}{B_n}\right)^2p_n(\cos\f_1|-1/c,-cq,a_3,a_4)  p_n(\cos\f_2|-1/c,-cq,a_3,a_4).
\eg
\end{eqnarray}

\section{A Three Parameter  Family of Operators}
We briefly outline a one parameter generalization of the operators 
$ \mathcal{K}_{a,c}$  of Section 6. The proofs are similar and will be omitted. 
We set 
\begin{eqnarray}\label{eqdefTab}
\bg
(T(a,b,r)  f)(\cos \t)= \qquad \qquad  \\
\qquad \qquad  h(\cos \t;a, b)(r^2;q)_\infty
\int_0^\pi \frac{w_H(\cos \f |q) f(\cos\f) \sin\f d\f}{h(\cos\f; re^{i\t},r e^{-i\t})h(\cos\f;a,b)}.
\eg
\end{eqnarray}
In the rest of this section all our operators are defined on $C[-1,1]$. 
It is clear that $T(a,b,r)$ are positive linear operators. 

\begin{thm}
	The operators $\{T(a,b,r);-1<r<1\}$ have the properties as follow. 
	
\noindent  $\textup{(a)}$ They   
form a multiplicative semigroup, that is,  $T(a,b,r)T(a,b,s)=T(a,b,rs)$.\\	
\noindent   $\textup{(b)}$ $T(a,b,r)$ tends to the 
identity operator $I$	as $r \to 1^-$. \\
\noindent $\textup{(c)}$  The eigenvalue functions of $T(a,b,r)$   
are $h(\cos \t;a,b)H_n(x \mid q)$ with eigenvalues $\lambda_n=r^n$.  
\noindent $\textup{(d)}$ We have operators $\mathcal{ B}_q(a,b)$ such that $\mathcal{ B}_q (a,b)T(a,b,r)=T(a,b,rq^{-1/2})$,  where\\
 $$\mathcal{ B}_q(a,b)f= \frac{h(\cos \t;a,b)[(1-q^{-1/2}az)(1-q^{-1/2}bz)\breve{f}(q^{1/2}z) - z^2(1-q^{-1/2}a/z)(1-q^{-1/2}b/z)\breve{f}(q^{-1/2}z)]}{h(\cos \t;q^{-1/2}a,q^{-1/2}b)(q^{3/4}-q^{-1/4})(z^2-1)/2}$$ with  $\breve{f}(z)=f(\cos \t)$ and $z=e^{i\t}.$\\
 In fact, $\mathcal{ B}_q(a,b)$ also has the representation $$\mathcal{ B}_q(a,b)f=\frac{1}{g(x;a,b)}\D (g(.;a,b)f),$$
 where\\
  $$ g(\cos \t;a,b)=\frac{(-q^{1/4}e^{i\t},-q^{1/4}e^{-i\t};q^{1/2})_\infty}{h(\cos \t;a,b)} \quad \text{ and  $\D$ is Askey-Wislon operator.}$$
\end{thm}

The operators $T(a,b,r)$ are  compact because $T(a,b,r)$ is 
the limit of finite rank operators. Moreover $T(a,b,r)$ is not 
invertible for $-1<r<1$ because $\lambda=0$ belongs to the 
spectrum.

 We next apply the operators $T(a,b,r)$ to derive transmutation 
 relation and bilinear formulas.  
 
\begin{thm}\label{tabrakpo}
The operator $T(t_1,t_2,r)$ maps  an Askey--Wilson 
polynomial  to a similar one: 
\begin{eqnarray} \label{eqGAWTrans}
T(t_1,t_2,r)p_n(.;t_1,t_2,t_3,t_4)=\frac{h(x;t_1,t_2)(r^2t_1t_2q^n;q)_\infty}{h(x;t_1r,t_2r)(t_1t_2q^n;q)_\infty}p_n(x;t_1r,t_2r,t_3/r,t_4/r)r^n
\end{eqnarray}
\end{thm}
The proof uses the Askey--Wilson integral \eqref{eqAWI}  
and will be omitted. 

As in Section 6, the transmutation relation \eqref{eqGAWTrans}
 will lead a bilinear formula via the technique  developed  in 
 \cite{Ism}. The proof uses the orthogonality of the Askey--Wilson 
 polynomials \eqref{eqawpo}

As in Section 6, we define three sequences of $\{a_n\}$ $\{b_n\}$ and $\{c_n\}$,
\begin{eqnarray}
\bg
\notag
M_n(t_1,t_2,t_3,t_4)  =\frac{2\pi(t_1t_2t_3t_4q^{2n};q)_{\infty}(t_1t_2t_3t_4q^{n-1};q)_{n}}{(q^{n+1};q)_{\infty}\prod_{1\le j<k\le 4}(t_jt_kq^{n};q)_{\infty}}\\
a_n=M_n(t_1r,t_2r,t_3/r,t_4/r)  \quad \quad b_n= M_n(t_1,t_2,t_3,t_4) \quad \quad 
c_n=\frac{(r^2t_1t_2q^n;q)_\infty}{(t_1t_2q^n;q)_\infty}r^n.
\eg
\end{eqnarray}

Using the method of proof in Section 6 we prove that: 
\begin{eqnarray}
\bg
\int_0^\pi W_0(\cos \t \mid t_1r, t_2r, t_3/r, t_4/r)
\int_0^\pi \frac{(r^2;q)_\infty w_H(\cos \f_1 \mid q)p_n(\cos \f_1 \mid t_1, t_2,t_3 ,t_4)\sin \f_1}{h(\cos \f_1;re^{i\t},re^{-i\t})h(\cos \f_1;t_1,t_2)}\\
\int_0^\pi \frac{(r^2;q)_\infty w_H(\cos \f_2 \mid q)p_m(\cos \f_2 \mid t_1, t_2,t_3 ,t_4)\sin \f_1}{h(\cos \f_2;re^{i\t},re^{-i\t})h(\cos \f_2;t_1,t_2)}d\f_1 d\f_2 d\t=a_nc_n^2 \delta_{m, n},
\eg
\end{eqnarray}
where 
$$W_0(\cos \t \mid t_1r, t_2r, t_3/r, t_4/r)=w(\cot \t;t_1r,t_2r,t_3/r,t_4/r)h^2(\cos\t;t_1r,t_2r).$$ Moreover we conclude that 
 
\begin{eqnarray}
\frac{a_nc_n^2}{b_n}p_n(\cos \f_2;t_1, t_2, t_3, t_4)=\int_0^\pi k_0(\cos \f_1,\cos\f_2)p_n(\cos\f_1;t_1,t_2,t_3,t_4)d\f_1,
\end{eqnarray}
where
\begin{eqnarray}
\bg
\notag
k_0(\cos \f_1,\cos\f_2)= \frac{w_H(\cos \f_1 \mid q)(r^2;q)_\infty}{h(\cos\f_1;t_1,t_2)}\frac{w_H(\cos \f_2 \mid q)(r^2;q)_\infty}{h(\cos\f_2;t_1,t_2)}\\
\times \frac{\sin \f_1 \sin\f_2}{w(\cos \f_2;t_1,t_2,t_3,t_4)}\\
\times \int_{0}^{\pi} \frac{W_0(\cos \t \mid t_1r, t_2r, t_3/r, t_4/r)}{h(\cos\f_1;re^{i\t},re^{-i\t})h(\cos\f_2;re^{i\t},re^{-i\t})}d\t,
\eg
\end{eqnarray}
which leads to the   bilinear formula  
\begin{eqnarray}
\label{k0w}
\frac{k_0(\cos \f_1,\cos\f_2)}{w(\cos \f_1|t_1,t_2,t_3,t_4)}=\sum_{n=0}^\infty a_n\left(\frac{c_n}{b_n}\right)^2 p_n(\cos \f_1| t_1,t_2,t_3,t_4)p_n(\cos \f_2| t_1,t_2,t_3,t_4).
\end{eqnarray}

It must be  noted that the special case $b=q^{1/2}a$ of our operators  is very interesting because some of the expressions simplify a great deal. In this case the operators $B_q$ becomes 
\bea
\bg
\mathcal{ B}_q(a,q^{1/2}a)f=\frac{1}{(q^{3/4}-q^{-1/4})(z^2-1)/2} \\
\times  \left[\left(\frac{1-az}{1-q^{-1/2}/z}\right)\breve{f}(q^{1/2}z)-z^2\left(\frac{1-a/z}{1-aq^{-1/2}z}\right)\breve{f}(q^{-1/2}z)\right],
\eg
\eea
 while the $T$ operator becomes 
\begin{eqnarray} 
\bg
(T(a,q^{1/2}a,r)  f)(\cos \t)= \qquad \qquad  \\
\qquad \qquad  (ae^{i\t},ae^{-i\t};q^{1/2})_\infty(r^2;q)_\infty
\int_0^\pi \frac{w_H(\cos \f |q) f(\cos\f) \sin\f d\f}{h(\cos\f; re^{i\t},r e^{-i\t})(ae^{i\f},ae^{-i\f};q^{1/2})_\infty}.
\eg
\end{eqnarray}

\noindent M. E. H. I, 
  Department of Mathematics\\
University of Central Florida, Orlando, Florida 32816\\
email: ismail@math.ucf.edu
 
 \bigskip 
 \noindent K. Zhou,
School of Mathematics and Statistics\\ Central South University, Changsha, Hunan 410083, P.R. China\\
and Department of Mathematics,\\ University of Wisconsin, Madison, WI 53706-1388, USA\\
email: kzhou64@wisc.edu

 \end{document}